\documentclass[12pt]{extarticle}
\usepackage{amsmath, amsthm, amssymb, hyperref, color}
\usepackage{graphicx}
\usepackage[all]{xypic}
\usepackage{verbatim}
\usepackage{tikz}
\usepackage{booktabs}
\usepackage{cite}
\usepackage{caption}
\usepackage{subcaption}
\usepackage{listings}
\oddsidemargin \evensidemargin
\usepackage{makecell}
\usepackage{array}
\newcolumntype{?}{!{\vrule width 1pt}}

\graphicspath{ {figures/} }

\newtheorem{theorem}{Theorem}
\numberwithin{theorem}{section}
\newtheorem{proposition}[theorem]{Proposition}
\newtheorem{lemma}[theorem]{Lemma}
\newtheorem{corollary}[theorem]{Corollary}

\newtheorem{problem}[theorem]{Problem}
\newtheorem{remark}[theorem]{Remark}
\newtheorem{example}[theorem]{Example}
\newtheorem{conjecture}[theorem]{Conjecture}
\theoremstyle{definition}

\newcommand{\PP}{\mathbb{P}}
\newcommand{\RR}{\mathbb{R}}
\newcommand{\QQ}{\mathbb{Q}}
\newcommand{\CC}{\mathbb{C}}

\newcommand{\NN}{\mathbb{N}}

\DeclareMathOperator{\Vor}{Vor}

\DeclareMathOperator*{\argmin}{arg\,min}
\newcommand{\SymRR} {\mathcal{S}}
\DeclareMathOperator{\Jac}{Jac}

\date{}
 
\title{\textbf{Voronoi Cells of Varieties}
\author{Diego Cifuentes, Kristian Ranestad, \\ Bernd Sturmfels and
Madeleine Weinstein}}

\begin{document}
\maketitle

 \begin{abstract}
 \noindent
Every real algebraic variety determines a Voronoi decomposition
of its ambient Euclidean space. Each Voronoi cell  is a convex 
semialgebraic set in the normal space of the variety at a point.
We compute the algebraic boundaries of these Voronoi~cells.
 \end{abstract}

\section{Introduction}

Every finite subset $X$ of $\RR^n$ defines a Voronoi decomposition of the ambient Euclidean space.
The \emph{Voronoi cell} of a point $y\in X$ consists of all points whose closest point in $X$ is~$y$, i.e.
\begin{equation}
\label{eq:voronoidef}
\Vor_X(y) \,\,\,:= \,\,\, \bigl\{\,u\in \RR^n: y \in \argmin_{x\in X} \|x-u\|^2 \,\bigr\}.
\end{equation}
This is a convex polyhedron with at most $|X|-1$ facets.
The study of these cells, and how they depend on $X$, is ubiquitous in
computational geometry and its numerous applications.

In what follows we assume that $X$ is a real algebraic variety of codimension $c$ and that $y$ is a smooth point on $X$.
The ambient space is $\RR^n$ with its Euclidean metric.
The Voronoi cell $\,\Vor_X(y)\,$ is a convex semialgebraic set of dimension~$c$.
It lives in the {\em normal space} 
$$ N_X(y) \,\,= \,\, \bigl\{ u \in \RR^n \, : \, u-y  \,\,\,
\hbox{is perpendicular to the tangent space of $X$ at $y$} \bigr\}. $$
The topological boundary of $\Vor_X(y)$ 
in $N_X(y)$ is denoted by $\partial \Vor_X(y)$.
It  consists of the points in $X$ that have at least two closest points in $X$, including~$y$.
In this paper we study the {\em algebraic boundary} $\partial_{\rm alg} \! \Vor_X(y)$.
This is the hypersurface in the complex affine space $ N_X(y)_\CC \simeq \CC^{c}$
obtained as the Zariski closure of $\partial \Vor_X(y)$ over the field of definition of $X$.
The degree of this hypersurface is denoted $\delta_X(y)$ and called the 
{\em Voronoi degree} of $X$ at $y$.
If  $X$ is irreducible and $y$ is a general point on $X$,
then this degree does not depend on $y$.

\begin{example}[Surfaces in 3-space] \rm
Fix a general inhomogeneous polynomial $f \in \QQ[x_1,x_2,x_3]$
of degree $d \geq 2$ and let $X = V(f)$ be its surface in $\RR^3$.
The normal space at a general point $y \in X$ is the line
$\,N_X(y) = \{ y + \lambda (\nabla f) (y) \,:\, \lambda \in \RR \}$.
The Voronoi cell $\Vor_X(y)$ is a line segment (or ray) 
in $N_X(y)$ that contains the point $y$.
The boundary $\partial \Vor_X(y)$ consists of $\leq 2$ points from among the 
zeros of an irreducible polynomial in $\QQ[\lambda]$.
We shall see that this polynomial has degree $d^3+d-7$.
Its complex zeros form the algebraic boundary $\partial_{\rm alg} \! \Vor_X(y)$.
Thus, the Voronoi degree of the surface~$X$ is $d^3+d-7$.
For example, let $d=2$ and fix $\,y=(0,0,0)\,$ and
$\,f = x_1^2 + x_2^2 + x_3^2 -3 x_1 x_2 -5 x_1 x_3-7x_2 x_3+x_1+x_2+x_3$.
Then  $\partial_{\rm alg} \! \Vor_X(y)$ consists of the three zeros of
$\, \bigl\langle \,u_1 - u_3, \,u_2-u_3, \,368 u_3^3+71 u_3^2-6u_3-1\, \bigr\rangle $.
The Voronoi cell is the segment $\,\Vor_X(y) = \{ (\lambda,\lambda,\lambda) \in \RR^3 \,:\,
-0.106526\ldots \leq \lambda \leq  0.12225\ldots \}$.
\end{example}

\begin{figure}[htb]
  \centering
  \includegraphics[width=195pt]{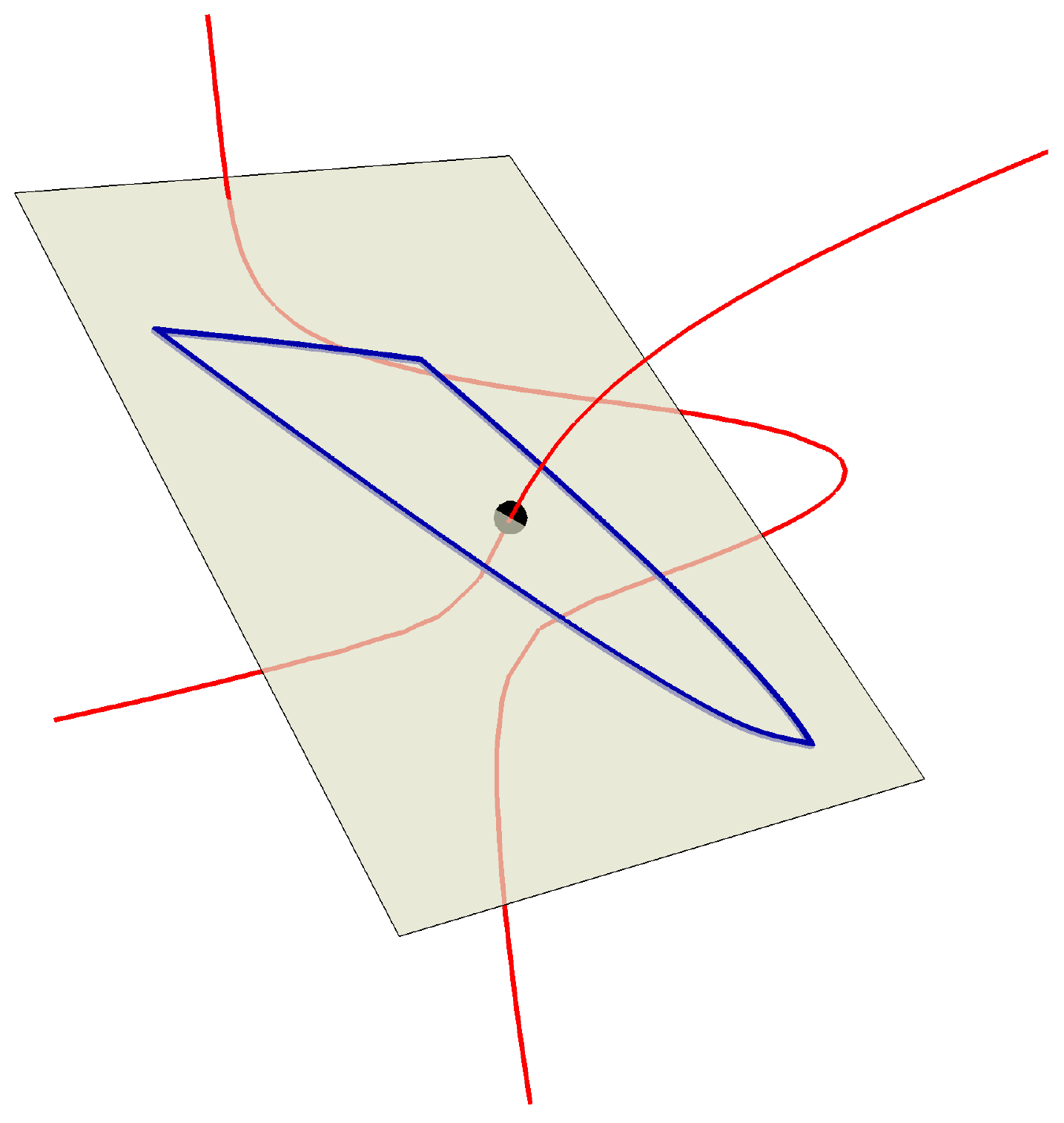} \qquad\qquad
  \includegraphics[width=188pt]{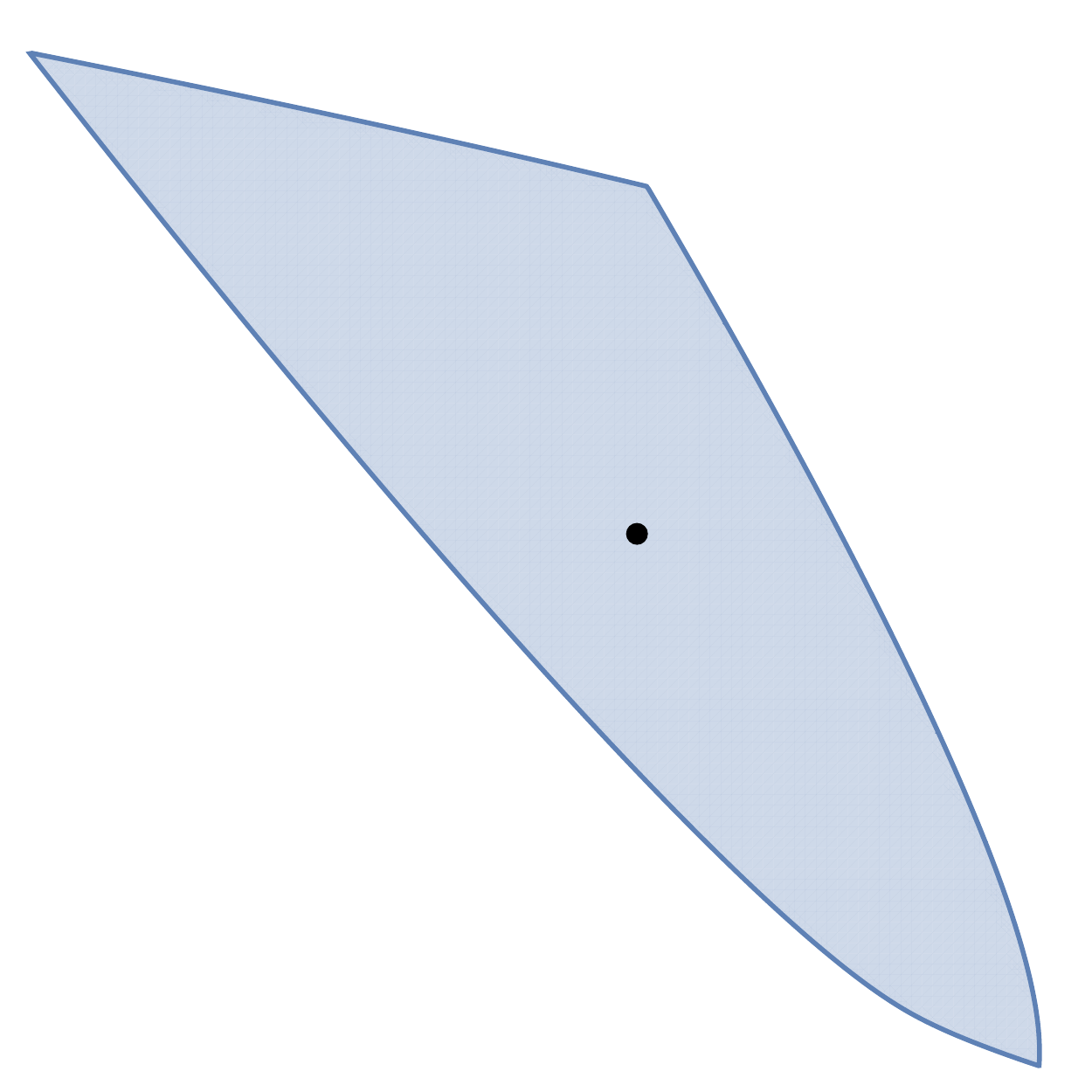}
  \caption{\label{fig:spacecurve} A quartic space curve, shown with the Voronoi cell in one of its normal planes.}
\end{figure}

\begin{example}[Curves in 3-space]\label{ex:spacecurve} \rm
Let $X $ be a general algebraic curve in $\RR^3$.
For $ y \in X$, the Voronoi cell $\Vor_X(y)$ is a convex set in the normal plane $N_X(y) \simeq \RR^2$.
Its algebraic boundary $\partial_{\rm alg} \! \Vor_X(y)$ is a plane curve
of degree $\delta_X(y)$. This  Voronoi degree can be expressed
in terms of the degree and genus of $X$.
Specifically, if $X$ is the intersection of two general quadrics in $\RR^3$, then the Voronoi degree is $12$.
Figure~\ref{fig:spacecurve} shows one such quartic space curve~$X$ together with the normal plane at a point $y \in X$.
The Voronoi cell $\Vor_X(y)$ is the planar convex region highlighted on the right.
Its boundary  is an algebraic curve of degree $\delta_X(y) = 12$.
\end{example}

 Voronoi cells of varieties belong to
the broader context of {\em metric algebraic geometry}.
We propose this name for the research trend that revolves around
articles like \cite{BKL, CAPT, CHS, DHOST, Ek, HW, OS}.
Metric algebraic geometry is concerned with
properties of real algebraic varieties that depend on a distance metric.
Key concepts include
the Euclidean distance degree \cite{BKL, DHOST}, distance function \cite{OS},
bottlenecks \cite{Ek}, reach,  offset hypersurfaces, medial axis \cite{HW}, and cut locus \cite{CHS}.

We study the Voronoi decomposition to answer the question for any point in ambient space, ``What point on the variety $X$ am I closest to?'' Another question one might ask is, ``How far do we have to get away from $X$ before there is more than one answer to the closest point question?'' The union of the boundaries of the Voronoi cells is the locus of points in $\RR^n$ that have more than one closest point on $X$. This set is called the \emph{medial axis} (or \emph{cut locus}) of the variety. The distance from the variety to its medial axis, which is the answer to the ``how far'' question, is called the \emph{reach} of $X$. This quantity is of interest, for example, in topological data analysis, as it is the main quantity determining the density of sample points needed to compute the persistent homology of $X$.  We refer to 
\cite{BKSW, BO, DEHH} for recent progress on sampling at the interface
of  topological data analysis with metric algebraic geometry.
The distance from a point $y$ on $X$ to the variety's medial axis could be considered the {\em local reach} of $X$. Equivalently, this is the distance from $y$
 to the boundary of its Voronoi cell $\Vor_X(y)$.

The present paper is organized as follows.
In Section~\ref{s:2} we describe the exact symbolic computation
of the Voronoi boundary at $y$ from the equations that define $X$.
We present a Gr\"obner-based algorithm
whose input is $y$ and the ideal of $X$ and whose output is
the ideal defining $\partial_{\rm alg} \! \Vor_X(y)$.
In Section~\ref{s:3} we consider the case when $y$ is a low rank matrix and $X$
is the variety of these matrices.
Here, the Eckart-Young Theorem yields an explicit description of $\Vor_X(y)$ in terms of the spectral norm.
Section~\ref{s:4} is concerned with inner approximations of the Voronoi cell
$\Vor_X(y)$ by spectrahedral shadows.
This is derived from the Lasserre hierarchy in polynomial optimization.
In Section~\ref{s:5} we present formulas for the degree
of the Voronoi boundary $\partial_{\rm alg} \! \Vor_X(y)$ when 
$X,y$ are sufficiently general and ${\rm dim}(X)\leq 2$.
These formulas are proved in Section~\ref{s:6} using tools from intersection theory in algebraic geometry.

\section{Computing with Ideals}\label{s:2}

In this section we describe Gr\"obner basis methods for finding the
Voronoi boundaries of a given variety.
We start with an ideal
$I = \langle f_1,f_2,\ldots,f_m \rangle$ in $\QQ[x_1,\ldots,x_n]$ whose
real variety $X = V(I) \subset \RR^n$ is assumed to be nonempty.
One often further assumes that $I$ is real radical and prime, so that 
$X_\CC$ is an irreducible variety in $\CC^n$ whose real points are Zariski dense.
Our aim is to compute the Voronoi boundary of a given point $y \in X$.
In our examples, the coordinates of the
point $y$ and the coefficients of the polynomials $f_i$ are rational numbers.
Under these assumptions, the following computations are done
 in polynomial rings over $\QQ$.

Fix the polynomial ring $R = \QQ[x_1,\ldots,x_n,u_1,\ldots,u_n]$
where $u = (u_1,\ldots,u_n)$ is an additional unknown point.
The {\em augmented Jacobian} of $X$ at $x$ is the following matrix of size
$(m+1)\times n$ with entries in $R$. 
It contains the $n$ partial derivatives of the $m$ generators of~$I$:
\begin{equation*} 
J_I(x,u) \quad := \quad
\begin{bmatrix}
    u-x \\
    (\nabla f_1)(x) \\
    \vdots \\
    (\nabla f_m)(x)
\end{bmatrix} 
\end{equation*}
Let $N_I$ denote the ideal in $R$ generated by $I$ and  the $(c+1) \times (c+1)$ minors
of the augmented Jacobian $J_I(x,u)$, where $c$ is the codimension of the given variety $X \subset \RR^n$.
The ideal $N_I$ in $R$ defines a subvariety of dimension $n$ in $\RR^{2n}$,
namely the {\em Euclidean normal bundle} of $X$. 
Its points are pairs $(x,u)$ where $x$ is a point in  $X$ and $u$ lies in the normal space of $X$ at $x$.

\begin{example}[Cuspidal cubic]
\label{ex:cuspidal} \rm 
Let $n=2$ and $I = \langle \,x_1^3 - x_2^2\,\rangle$, 
so $X = V(I) \subset \RR^2$ is a cubic curve with a cusp at the origin.
The ideal of the Euclidean normal bundle of $X$ is
\begin{align*} 
 N_I \,\,=\,\, \bigl\langle \,x_1^3-x_2^2 \, , \,
 \det \! \left(\begin{smallmatrix} u_1 {-} x_1 \, &  \, u_2 {-} x_2 \\ 3 x_1^2 & -2 x_2 \end{smallmatrix}\right)
 \, \bigr\rangle. 
\end{align*}
\end{example}

Let $N_I(y)$ denote the linear ideal that is obtained from $N_I$ by replacing
the unknown point $x$ by the given point $y \in \RR^n$.
For instance, for $y= (4,8)$ we obtain $\,N_I(y) = \langle u_1 + 3 u_2 - 28 \rangle$.
We now define the {\em critical ideal} of the variety $X$ at the point $y$ as
\begin{equation*} 
 C_I(y) \quad = \quad \, I \, + \, N_I \, + \, N_I(y) \,+\,
\langle\, \|x-u\|^2 - \|y-u\|^2 \rangle 
 \quad \subset \quad R. 
\end{equation*}
The variety of $C_I(y)$ consists of  pairs $(u,x)$ such that
$x$ and $y$ are equidistant from $u$ and both are critical points of the
distance function from $u$ to $X$.
The {\em Voronoi ideal} is the following ideal in $\QQ[u_1,\ldots,u_n]$. It
is obtained  from the critical ideal by saturation and elimination:
\begin{equation}
\label{eq:voonoiideal} \Vor_I(y) \quad =\quad  \bigl( \,
C_I (y) \,:\, \langle x-y \rangle^\infty \bigr) \,\cap \, \QQ[u_1,\ldots,u_n]. 
\end{equation}

The geometric interpretation of each step in
our construction implies the following result:

\begin{proposition} \label{prop:algorithm}
The affine variety in $\,\CC^n$  defined by the Voronoi ideal $\,\Vor_I(y)$
contains the algebraic Voronoi boundary 
$\partial_{\rm alg} \! \Vor_X(y)$ of the given real variety $X$ at its point $y$.
  \end{proposition}

\begin{example} \rm
For the point $y=(4,8)$ on the
cuspidal cubic $X$ in Example~\ref{ex:cuspidal}, we have
$N_I(y) = \langle u_1+3 u_2-28 \rangle $.
Going through the steps above, we find that the Voronoi ideal~is
$$ 
\Vor_I(y)  \,\,\,=\,\,\,\langle u_1-28,u_2 \rangle 
\,\,\cap \,\, \langle u_1+26,u_2-18 \rangle
\,\,\cap \,\, \langle u_1+3 u_2-28,\,27 u_2^2-486 u_2+2197 \rangle. 
$$
The third component has no real roots and is hence extraneous.
The Voronoi boundary consists of two points: 
$\,\partial \Vor_X(y) = \{(28,0), (-26,18)\}$.
The Voronoi cell $\Vor_X(y)$ is the line segment connecting these points.
This segment is shown in green in Figure~\ref{fig:cuspidalcubic}.
Its right endpoint $(28,0)$ is equidistant from $y$ and the point $(4,-8)$.
Its left endpoint $(-26,18)$ is equidistant from $y$ and the point $(0,0)$,
whose Voronoi cell is discussed in Remark \ref{rem:singular}.

The cuspidal cubic $X$ is very special.
If we replace $X$ by a general cubic (defined over $\QQ$) in the affine plane, then $\Vor_I(y) $ is generated 
modulo $N_I(y)$ by an irreducible polynomial of degree eight in
$\QQ[u_2]$.
Thus, the expected Voronoi degree of (affine) plane cubics is 
$\delta_X(y) = 8$.
\end{example}

\begin{figure}[htb]
  \centering
  \includegraphics[width=270pt]{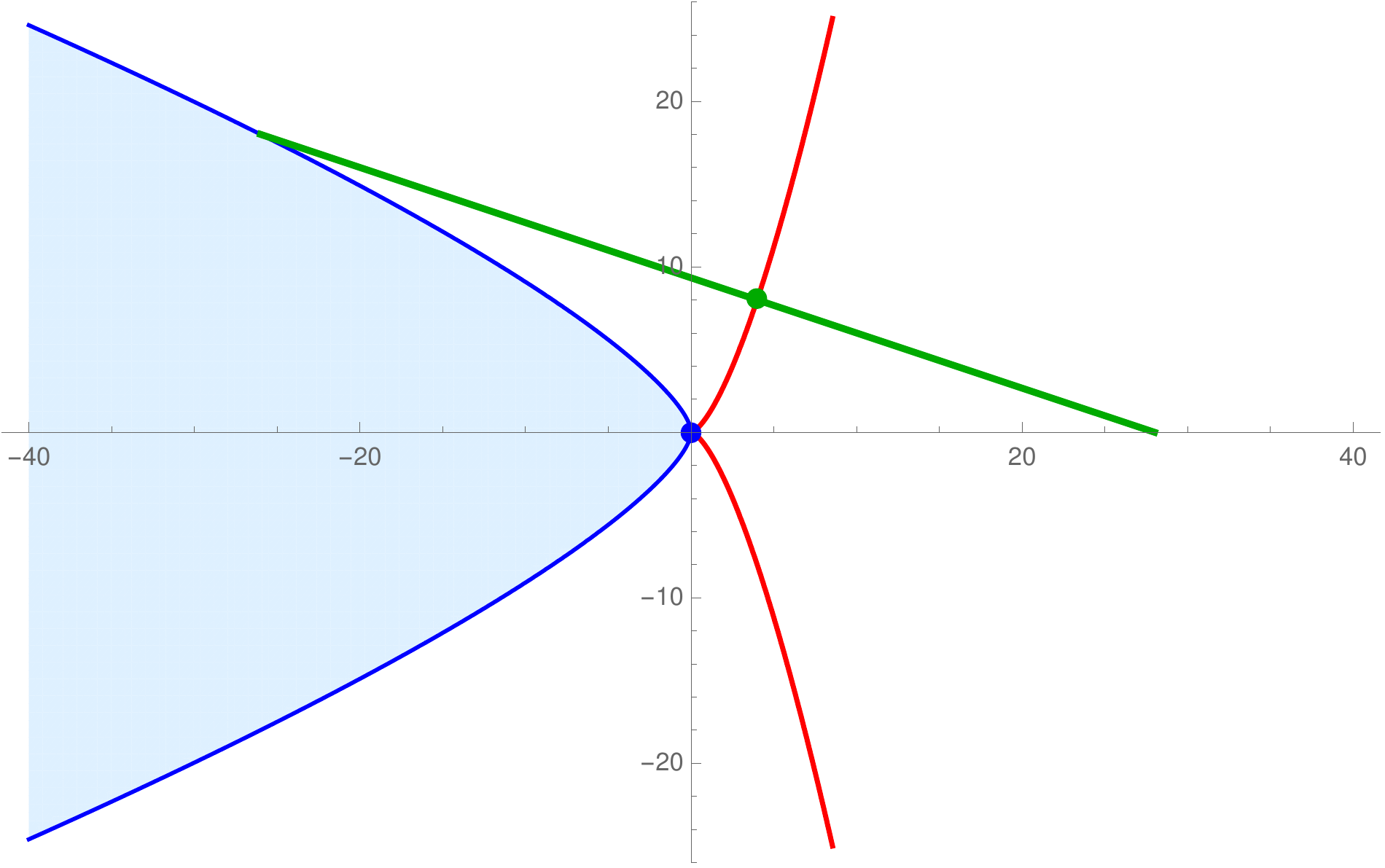}
  \caption{\label{fig:cuspidalcubic} The cuspidal cubic is shown in red.
  The Voronoi cell of a smooth point is a green line segment.
  The Voronoi cell of the cusp is the convex region bounded by the blue curve.}
\end{figure}

\begin{remark}[Singularities] \label{rem:singular} \rm
Voronoi cells at singular points can be computed with the same procedure as above.
However, these Voronoi cells generally have higher dimensions.
For an illustration, consider the cuspidal cubic, and let $y\!=\!(0,0)$ be the cusp.
A Gr\"obner basis computation  yields the Voronoi boundary
$\,27u_2^4+128u_1^3+72u_1u_2^2+32u_1^2+u_2^2+2u_1 $.
The Voronoi cell is the two-dimensional convex region bounded by this quartic,
shown in blue in Figure~\ref{fig:cuspidalcubic}.
The Voronoi cell might also be empty at a singularity.
This happens for instance for $V(x_1^3 + x_1^2 - x_2^2)$, which has an ordinary double point at $y\!=\!(0,0)$.
In general, the cell dimension depends on both the embedding dimension and the branches of the singularity.
\end{remark}

In this paper we restrict ourselves to Voronoi cells at points $y$
that are nonsingular in the given variety $X = V(I)$.  
Proposition~\ref{prop:algorithm} gives an algorithm for computing
the  Voronoi ideal $\,\Vor_I(y)$. We implemented it in {\tt Macaulay2} \cite{M2}
and experimented with numerous examples. 
For small enough instances, the computation terminates and we obtain the defining polynomial of
the Voronoi boundary $\partial_{\rm alg} \! \Vor_X(y)$.
This polynomial is unique modulo the linear ideal of the normal space $N_I(y)$.
For larger instances, we can only compute the degree of 
$\partial_{\rm alg} \! \Vor_X(y)$ but not its equation.
This is done by working over a finite field and adding  $c-1$ random linear equations in $u_1,\ldots,u_n$
in order to get a zero-dimensional polynomial system. 

Our experiments were most extensive for the case of hypersurfaces $(c=1)$.
We sampled random polynomials $f$ of degree $d$ in  $\QQ[x_1,\ldots,x_n]$, both inhomogeneous and homogeneous. 
These were chosen among those that vanish at a
preselected point $y$ in $\QQ^n$.
In each iteration, the resulting Voronoi ideal $\Vor_I(y)$ 
from  (\ref{eq:voonoiideal}) was found to be zero-dimensional.
In fact, $\Vor_I(y)$ is a maximal ideal in $\QQ[u_1,\ldots,u_n]$, 
and $\delta_X(y)$ is the degree of the associated field extension.
We summarize our results in Tables~\ref{tab:voronoi} and~\ref{tab:voronoihom},
 and we extract conjectural formulas.

\begin{table}[htb]
  \centering
  \setlength{\tabcolsep}{3pt}
    \begin{tabular}{c|*{10}{c}}
    \toprule
    $n\backslash d$ & \;2\; & 3 & 4 & 5 & 6 & 7 & 8 & $
    \delta_{X}(y) = 
    {\rm degree}(\Vor_{\langle f \rangle}(y))$    \\
    \midrule
    1 & 1 & 2 & 3& 4 & 5 & 6 & 7 
   & $d{-}1$  \\
    2 & 2 & 8 & 16 & 26 & 38 & 52 & 68 
   & $d^2{+}d{-}4$  \\
    3 & 3 & 23 & 61 & 123 & 215 & 343& 
    & $d^3{+}d{-}7$  \\
    4 & 4 & 56 & 202 & 520 & 1112  & & 
    & $d^4{-}d^3{+}d^2{+}d{-}10$  \\
    5 & 5 & 125 & 631 & & & & 
    & $d^5{-}2d^4{+}2d^3{+}d{-}13$ \\
    6 & 6 & 266 & 1924 & & & & 
    & $d^6{-}3d^5{+}4d^4{-}2d^3{+}d^2{+}d{-}16$  \\
    7 & 7 & 551 & & & & & 
    & $d^7{-}4d^6{+}7d^5{-}6d^4{+}3d^3{+}d{-}19$ \\
    \bottomrule
    \end{tabular}%
      \caption{\label{tab:voronoi}
      The Voronoi degree of an inhomogeneous polynomial $f$ of degree $d$ in $\RR^n$.}
\end{table}

\begin{table}[htb]
  \centering
  \setlength{\tabcolsep}{3pt}
    \begin{tabular}{c|*{10}{c}}
    \toprule
    $n\backslash d$ & \;2\; & 3 & 4 & 5 & 6 & 7 & 8 & 
    $  \delta_{X}(y) = {\rm degree}(\Vor_{\langle f \rangle}(y))$ \\
        \midrule
    2 & 2 & 4 & 6& 8 & 10 & 12 & 14 
   & $2d{-}2$ \\
    3 & 3 & 13 & 27 & 45 & 67 & 93 & 123
    & $2d^2{-}5$ \\
    4 & 4 & 34 & 96 & 202 & & &
    & $2d^3{-}2d^2{+}2d{-}8$\\
    5 & 5 & 79 & 309 & & & & 
& $2d^4{-}4d^3{+}4d^2{-}11$ \\
    6 & 6 & 172 & & & & & 
& $2d^5{-}6d^4+8d^3{-}4d^2{+}2d{-}14$ \\
    7 & 7 & 361 & & & & & 
& $2d^6{-}8d^5{+}14d^4{-}12d^3{+}6d^2{-}17$ \\
    \bottomrule
    \end{tabular}%
        \caption{\label{tab:voronoihom}
      The Voronoi degree of a homogeneous polynomial $f$ of degree $d$ in $\RR^n$.}
\end{table}

\begin{conjecture}
The Voronoi degree of a generic hypersurface of degree $d$ in $\RR^n$~equals
$$
    (d-1)^n + 3(d-1)^{n-1} + \tfrac{4}{d-2}((d-1)^{n-1}-1) - 3n.
$$
The Voronoi degree of the cone of a generic homogeneous polynomial of degree $d$ in $\RR^n$~is
$$  2(d-1)^{n-1} + \tfrac{4}{d-2}((d-1)^{n-1}-1) - 3n + 2. $$
\end{conjecture}

\smallskip

We shall prove both parts of this conjecture for $n \leq 3$ in Section~\ref{s:6},
where we develop the geometric theory
of Voronoi degrees of low-dimensional varieties.
The case $d=2$ was analyzed in~\cite[Proposition~5.8]{CHS}.
In general,  for $n \geq 4$ and $d\geq 3$, the problem is still open.

\section{Low Rank Matrices}\label{s:3}

There are several natural norms on  the space $ \RR^{m \times n}$
of real $m \times n$ matrices.
We focus on two of these norms.
First, we have the  {\em Frobenius norm}
$\,\|U\|_F := \sqrt{\sum_{ij} U_{ij}^2}$.
And second, we have the {\em spectral norm} $\,\|U\|_2 := \max_i \sigma_i(U)\,$ which 
extracts the largest singular value.

\smallskip

Let $X $ denote the variety of real $m \times n$ matrices of rank $\leq r$.
Fix a rank $r$ matrix $V $ in $ X$. This is a nonsingular point in $X$.
We consider  the Voronoi cell  $\Vor_X(V)$ with respect to the
Frobenius norm. This is consistent with the previous sections
because the Frobenius norm agrees
with Euclidean norm on $\RR^{m \times n}$.
This identification will no longer be valid after Remark
\ref{rem:instructive} when we restrict to  the subspace of symmetric matrices.

\smallskip

Let $U \in \Vor_X(V)$, i.e.~the closest point to~$U$ in 
the rank $r$ variety $X$ is the matrix~$V$.
By the Eckart-Young Theorem, the matrix $V$ is derived from $U$ by
computing the singular value decomposition $ \,U = \Sigma_1 \, D \, \Sigma_2$.
Here $\Sigma_1$ and $\Sigma_2$ are orthogonal matrices
of size $m \times m $ and $n \times n$ respectively, and $D$ is a nonnegative
diagonal matrix whose entries are the singular values.
Let $D^{[r]}$ be the matrix that is obtained from $D$ by replacing all
singular values except for the $r$ largest ones by zero.
Then, according to  Eckart-Young, we have
$\,V = \Sigma_1 \cdot D^{[r]} \cdot \Sigma_2$.

\begin{remark} \rm
The Eckart-Young Theorem works for both the Frobenius norm and the spectral norm.
This means that $\Vor_X(V)$ is also the Voronoi cell for the spectral norm.
\end{remark}

The following is the main result in this section.

\begin{theorem} \label{thm:determinantal}
Let $V$ be an $m \times n$-matrix of rank $r$.
The Voronoi cell $\Vor_X(V)$
 is congruent up to scaling to the unit ball in the spectral norm on the
space of $\,(m-r) \times (n-r)$-matrices.
\end{theorem}

Before we present the proof, let us first see why the statement makes sense.
The determinantal variety $X$ has dimension $rm+rn-r^2$ in an ambient space of dimension $mn$.
The dimension of the normal space at a point is the difference of these two numbers,
so it equals $(m-r)(n-r)$.
Every Voronoi cell is a full-dimensional convex body in the normal space.
Next consider the case $m=n$ and restrict to the space of diagonal matrices.
Now $X$ is the set of vectors in $\RR^n$ having at most $r$ nonzero coordinates.
This is a reducible variety with $\binom{n}{r}$ components, each a coordinate subspace.
For a general point $y$ in such a subspace, the Voronoi cell $\Vor_X(y)$ is a convex polytope.
It is congruent to a regular cube of dimension $n-r$, which is the unit ball in the $L^\infty$-norm on $\RR^{n-r}$.
Theorem~\ref{thm:determinantal}
describes the orbit of this picture under the
action of the two orthogonal groups on $\RR^{m \times n}$.
For example, consider the special case $n=3, r=1$.
Here, $X$ consists of the three coordinate axes in $\RR^3$. 
The Voronoi decomposition of this curve decomposes
$\RR^3$ into squares, each normal to a different point on the three lines.
The image of this picture under orthogonal transformations is 
the Voronoi decomposition of $\RR^{3 \times 3}$ associated with the
affine variety of rank $1$ matrices.
That variety has dimension $5$, and each Voronoi cell is a $4$-dimensional convex body in the normal space.

\begin{proof}[Proof of Theorem~\ref{thm:determinantal}]
The Voronoi cell is invariant under orthogonal transformations.
We may therefore assume that the matrix $V = (v_{ij})$ satisfies
$v_{11} \geq v_{22} \geq \cdots \geq v_{rr} = u > 0$ and
$v_{ij} = 0$ for all other entries.
The Voronoi cell of the diagonal matrix 
$V$ consists of matrices $U$ whose block-decomposition
into $r + (m-r)$ rows and $r + (n-r)$ columns satisfies
$$
\begin{pmatrix} I & 0 \\
0 & T_1 \end{pmatrix}
 \cdot  
 \begin{pmatrix} U_{11} & U_{12} \\
 U_{21} & U_{22} \end{pmatrix}
  \cdot
\begin{pmatrix} I & 0 \\
0 & T_2 \end{pmatrix}\,\,=\,\,
 \begin{pmatrix} V_{11} & 0 \\
 0 & V_{22} \end{pmatrix}.
$$
Here $V_{11} = {\rm diag}(v_{11},\ldots,v_{rr})$ agrees with the upper $r \times r$-block of $V$,
and $V_{22}$ is a diagonal matrix whose entries are bounded
above by $u$ in absolute value.
This implies  $U_{11} = V_{11}$,
$U_{12} = 0$, $U_{21} = 0$, and $U_{22}$ is
an arbitrary $(m-r) \times (n-r)$ matrix
with spectral norm at most $u$.
Hence the Voronoi cell of $V$ is congruent to the set of all such matrices $U_{22}$.
This convex body equals $u$ times the  unit ball in
 $\RR^{(m-r) \times (n-r)}$ under the  spectral norm.
\end{proof}

\begin{remark} \label{rem:instructive} \rm
It is instructive to compare the 
Voronoi degree with the \emph{Euclidean distance degree} (ED degree).
Assume  $m \leq n$ in Theorem~\ref{thm:determinantal}.
According to \cite[Example~2.3]{DHOST},
the ED degree of the determinantal
variety $X$ equals $\binom{m}{r}$. On the other hand, 
the Voronoi degree of $X$ is $2(m-r)$. Indeed, we have shown that
the Voronoi boundary is isomorphic to the hypersurface
$\,\{ {\rm det}(W W^T -  I_{m-r}) = 0\}$,
where $W$ is an $(m{-}r) \times (n{-}r)$ matrix of unknowns.
\end{remark}

Our problem becomes even more interesting when we restrict 
to matrices in a linear subspace. To see this, let
$X$ denote the variety of symmetric $n\times n$ matrices of rank $\leq r$.
We can regard $X$ either as a variety in the ambient matrix space $\RR^{n \times n}$
or in the space $\RR^{\binom{n+1}{2}}$ whose coordinates are the
upper triangular entries of a symmetric matrix. On the latter space we have 
both the {\em Euclidean norm} and the {\em Frobenius norm}. 
These are now different!

The Frobenius norm on $\RR^{\binom{n+1}{2}}$ is
the restriction of the Frobenius norm on $\RR^{n \times n}$ to the
subspace of symmetric matrices. 
For instance, if $n=2$, we identify the vector $(a,b,c)$ with the symmetric matrix
$\bigl(\begin{smallmatrix}a&b\\b&c\end{smallmatrix}\bigr)$.
The Frobenius norm is $\sqrt{a^2{+}2b^2{+}c^2}$, whereas the Euclidean norm is $\sqrt{a^2{+}b^2{+}c^2}$.
The two norms have dramatically different
properties with respect to low rank approximation.
The Eckart-Young Theorem remains valid for the Frobenius norm on
$\RR^{\binom{n+1}{2}}$, but this is not true for the Euclidean norm
(cf.~\cite[Example~3.2]{DHOST}). In what follows we elucidate this
by comparing the Voronoi cells with respect to the two norms.

\begin{figure}[htb]
  \centering
  \includegraphics[width=170pt]{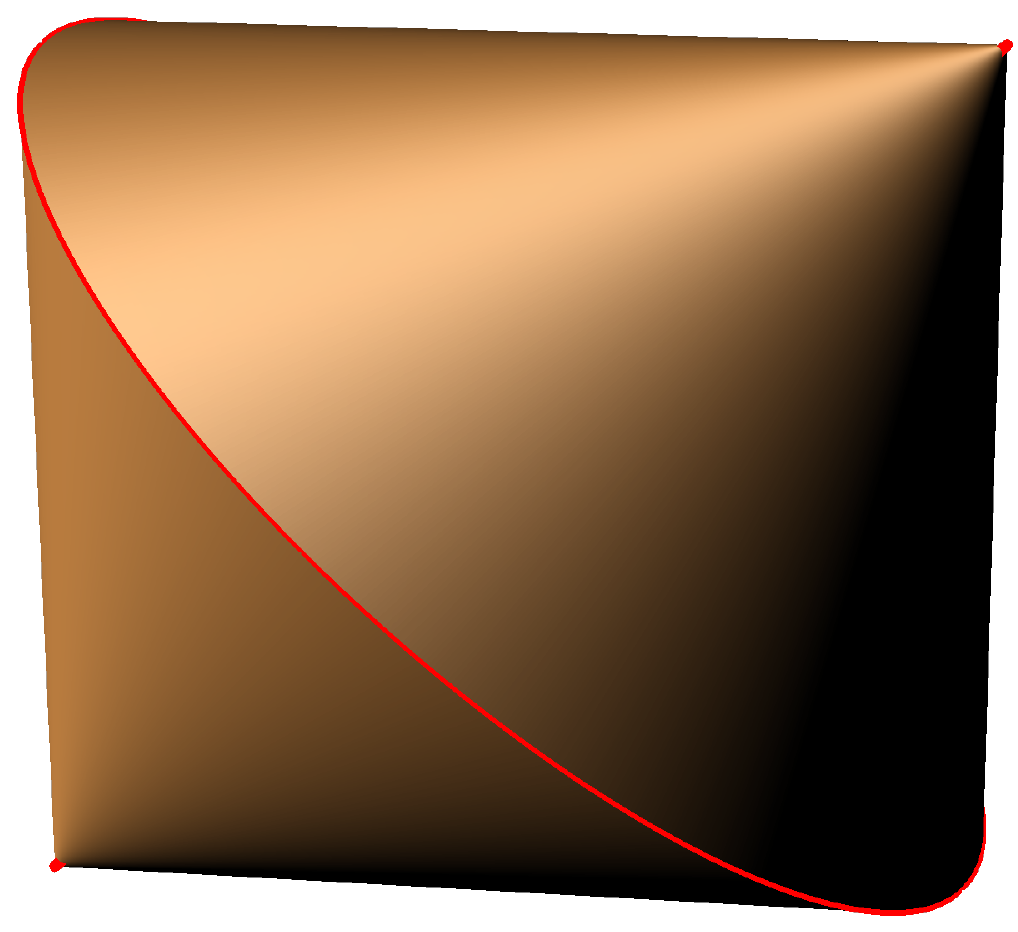} \qquad \qquad \qquad
  \includegraphics[width=170pt]{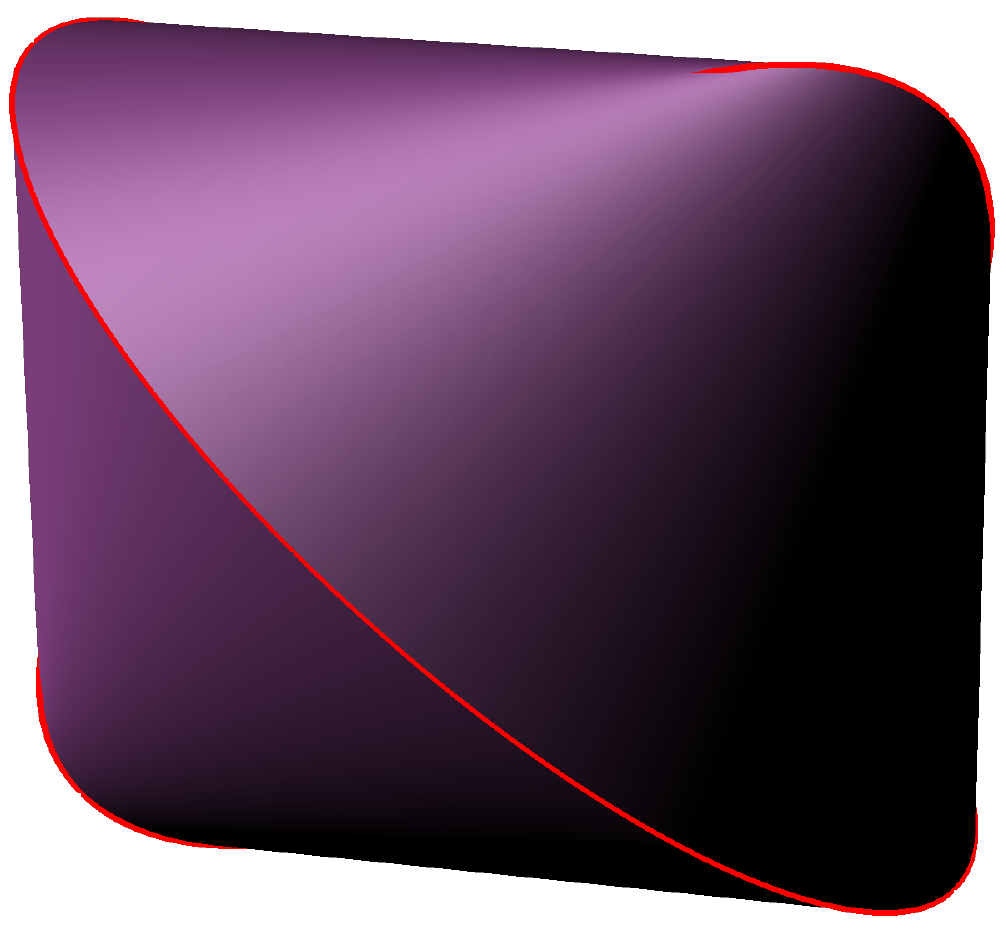}
  \caption{\label{fig:rank1matrices} The Voronoi cell of a symmetric $3 {\times} 3$ matrix of rank $1$  is a
   convex body of dimension $3$. It is shown for the Frobenius norm (left) and for
  the Euclidean norm (right).}
  \end{figure}

\begin{example} \label{ex:331} \rm Let $X$ be the variety of
symmetric $3\times 3$ matrices of rank $\leq 1$.
For the Euclidean metric,  $X$ lives in $\RR^6$.
For the Frobenius metric, $X$ lives in 
a $6$-dimensional  subspace of $\RR^{3\times 3}$.
Let $V$ be a regular point in $X$, i.e.~a symmetric $3 \times 3$ matrix of rank~$1$.
The normal space to $X$ at $V$ has dimension $3$.
Hence, in either norm, the Voronoi cell $\Vor_X(V)$  is a $3$-dimensional convex body.
Figure~\ref{fig:rank1matrices} illustrates these  bodies for our
two metrics.

For the Frobenius metric, the Voronoi cell is isomorphic to the set of matrices
$\left(\begin{smallmatrix}a&b\\b&c\end{smallmatrix}\right)$ with eigenvalues between $-1$ and~$1$.
This semialgebraic set is bounded by the surfaces defined by the singular quadrics
$\det\left(\begin{smallmatrix}a+1&b\\b&c+1\end{smallmatrix}\right)$ and 
$\det\left(\begin{smallmatrix}a-1&b\\b&c-1\end{smallmatrix}\right)$.
The Voronoi ideal is of degree~$4$, defined by the product of these two determinants (modulo the normal space).
The Voronoi cell is shown on the left in Figure~\ref{fig:rank1matrices}. 
It is the intersection of two quadratic cones.
The cell is the convex hull of the circle in which the two quadrics meet, together with the two vertices.

For the Euclidean metric, the Voronoi boundary at a generic point $V$
in $X$ is defined by an irreducible polynomial of degree~$18$ in $a,b,c$. 
In some cases, the Voronoi degree can drop.
For instance, consider the special rank $1$ matrix
$\,V =\left(\begin{smallmatrix}1&0&0\\0&0&0\\0&0&0\end{smallmatrix}\right)$.
For this point, the degree of the Voronoi boundary is only~$12$. 
This particular Voronoi cell is shown on the right in Figure~\ref{fig:rank1matrices}.
This cell is the convex hull of two ellipses, which are shown in red in the diagram.
\end{example}

\section{Spectrahedral Approximations of Voronoi Cells}\label{s:4}

Computing Voronoi cells of varieties is computationally hard.
In this section we introduce some tractable approximations to the Voronoi cell 
based on semidefinite programming (SDP).
More precisely, for a point $y\in X$ we will construct convex sets 
$\,\{S_X^d(y)\}_{d\geq 1}\,$ such that
\begin{align}\label{eq:inclusions}
  S_X^1(y) \,\subset \,S_X^2(y) \,\subset\, S_X^3(y)\, \subset\,\, \cdots \,\,\subset\, \Vor_X(y).
\end{align}
Here each  $S_X^d(y)$ is a spectrahedral shadow.
The construction is based on the sum-of-squares (SOS) hierarchy, also known as 
Lasserre hierarchy, for polynomial optimization problems~\cite{BPT}. This section is 
to be understood as a continuation of the studies undertaken in \cite{CAPT, CHS}.

Let $\SymRR^n$ denote the space of real symmetric $n\times n$ matrices.
Given $A,B\in \SymRR^n$, the notation $A\preceq B$ means that the 
matrix $B-A$ is positive semidefinite (PSD).
A \emph{spectrahedron} is the intersection of the cone of PSD matrices with an affine-linear space.
Spectrahedra are the feasible sets of SDP.
In symbols, a spectrahedron has the following form for some  $C_i\in \SymRR^n$:
\begin{align*}
  S \,\,\,:=\,\,\, \{ \,y \in \RR^m:  \,y_1 C_1 + \dots + y_m C_m \preceq C_0 \,\}.
\end{align*}
A \emph{spectrahedral shadow} is the image of an spectrahedron under an affine-linear map.
Using SDP one can efficiently maximize  linear functions over a spectrahedral shadow.

Our goal is to describe inner spectrahedral approximations of the Voronoi cells.
We first consider the case of \emph{quadratically defined varieties}. This is the setting
of \cite{CHS} which we now follow.
Let ${\bf f} = (f_1,\ldots,f_m)$ be a list of inhomogeneous quadratic polynomials in $n$ variables. We fix
$X = V({\bf f}) \subset \RR^n$ and we assume that $y$ is a nonsingular point in~$X$.
Let $A_i\in \SymRR^n$ denote the Hessian matrix of the quadric~$f_i$.
Consider the following spectrahedron:
\begin{align*}
  {\rm S}_{\bf f} \,\,:=\,\, \bigl\{ \lambda \in \RR^m : 
  \lambda_1 A_1 + \cdots + \lambda_m A_m  \preceq I_n \bigr\}.
\end{align*}
This was called the \emph{master spectrahedron} in \cite{CHS}.
Let $\Jac_{\bf f}$ be the Jacobian matrix of ${\bf f}$.
This is the $n\times m$ matrix with entries $\partial f_j /\partial x_i$.
The specialized Jacobian matrix $\Jac_{\bf f}(y)$ defines a linear map $\RR^m \to \RR^n$ whose range is the normal space of the variety at $y$. 
We define the set
\begin{align*}
  S^1_{\bf f}(y) \,:=\, y - \tfrac{1}{2} \Jac_{\bf f}(y) \cdot {\rm S}_{\bf f} \,\,\subset \,\, \RR^n.
\end{align*}
By construction, this is a  spectrahedral shadow. The following result was established in \cite{CHS}.

\begin{lemma}\label{thm:inclusions1}
The spectrahedral shadow $S^1_{\bf f}(y)$ is contained in the Voronoi cell $\,\Vor_X(y)$.
\end{lemma}

\begin{proof} We include the proof to better explain the situation.
  Let $u \in S^1_{\bf f}(y)$, so there exists $\lambda\in {\rm S}_{\bf f}$ with 
  $u = y - \frac{1}{2}\Jac(y)\lambda$.
  We need to show that $y$ is the nearest point from $u$ to the variety~$X$.
  Let $L(x,\lambda) = \|x-u\|^2 - \sum_i \lambda_i f_i(x)$ be the Lagrangian function, and let $L_\lambda(x)$ be the quadratic function obtained by fixing the value of $\lambda$.
  Observe that $L_\lambda(x)$ is convex, and its minimum is attained at~$y$.
  Indeed, 
  $\lambda\in {\rm S}_{\bf f}$ means that the Hessian of this function is positive semidefinite,
  and $\,u = y - \frac{1}{2}\Jac_{\bf f}(y)\lambda\,$ implies that $\nabla L_\lambda(y) = 0$.
  Therefore,
  \begin{align*}    \|y-u\|^2 \,=\,     L_\lambda(y)\, \,= \,\,\,
      \min_x L_\lambda(x)  \,  \leq  \, \min_{x\in X} \|x-u\|^2  \,\leq \, \|y-u\|^2.
  \end{align*}
  We conclude that $y$ is the minimizer of the squared distance function $x \mapsto \|x-u\|^2$ on $ X$.
\end{proof}

\begin{example}[\!\! {\cite[Ex~6.1]{CHS}}] \label{ex:twistedcubic} \rm
  Let $X\subset\RR^3$ be the twisted cubic curve, defined by 
  the two equations $x_2 {=} x_1^2$ and $x_3 {=} x_1x_2$.
  Both  $\Vor_X(0)$ and $S^1_{\bf f}(0)$ lie in the normal space at the origin, which is the plane $u_1{=}0$.
  The Voronoi cell is the planar convex set bounded   by the quartic curve
  $27 u_3^4{+}128 u_2^3{+}72 u_2 u_3^2{-}160 u_2^2{-}35 u_3^2{+}66 u_2 \!=\! 9$.
  The inner approximation $S^1_{\bf f}(0)$ is bounded by the parabola $2u_2 = 1 {-} u_3^2$.
  The two curves are tangent at the point $(0,\frac{1}{2},0)$.
\end{example}

\begin{figure}[h]
  \centering
  \includegraphics[width=150pt]{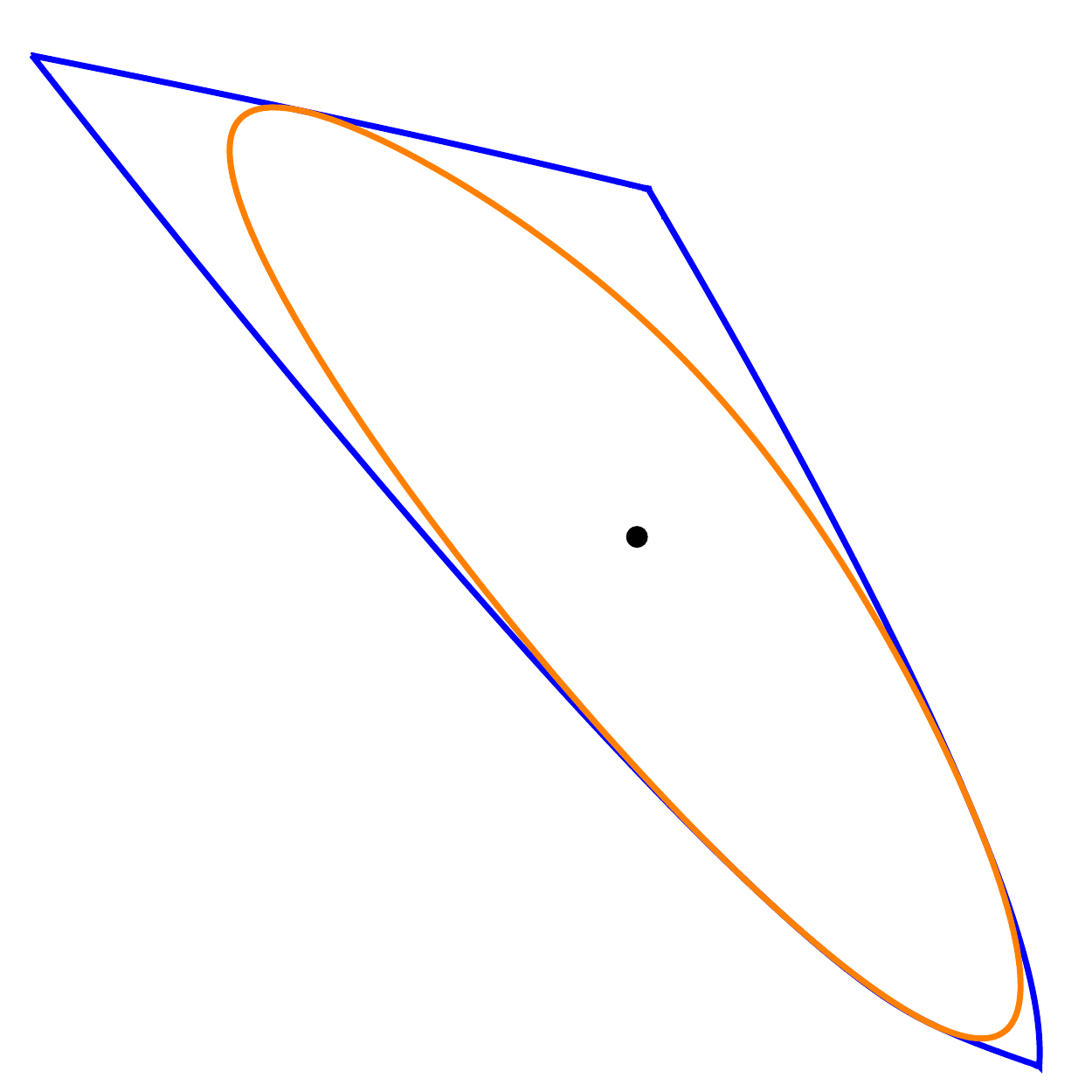}
  \caption{\label{fig:lasserre} The spectrahedral approximation 
  $S^1_{\bf f}(0)$ of the Voronoi cell $\Vor_X(0)$ shown in
  Figure~\ref{fig:spacecurve}. The boundaries of these two convex sets -- curves of degree $3$ and $12$ -- are tangent.}
\end{figure}

\begin{example} \rm
  Let $X\subset\RR^3$ be the quartic curve in Figure~\ref{fig:spacecurve}.  
  The Voronoi boundary is a plane curve of degree~$12$.
  The master spectrahedron ${\rm S}_{\bf f}$ is bounded by a cubic curve, as seen in~\cite[Example~5.2]{CHS}.
  The convex set $S_{\bf f}^1(y)$ is affinely isomorphic to 
  ${\rm S}_{\bf f}$, so it is also bounded by a cubic curve.
  Figure~\ref{fig:lasserre} shows the Voronoi cell and its inner spectrahedral approximation.
  Note that their boundaries are tangent.
\end{example}

The above examples motivate the following open problem. 

\begin{problem}
  Fix a quadratically defined variety $X=V({\bf f})$  in $\RR^n$. 
  Let $\partial_{\rm alg} \!\Vor_X(y)$ be the algebraic boundary of the Voronoi cell at $y \in X$ and let
  $\partial_{\rm alg} {S}_{\bf f}^1(y)$ 
  be its first spectrahedral approximation.
  Investigate the tangency behavior of these two hypersurfaces.
\end{problem}

This problem was studied in~\cite{CHS} for complete 
intersections of $n$~quadrics in~$\RR^n$.
Here, $X$~is a finite set, and it 
was proved in~\cite[Theorem~4.5]{CHS} that the Voronoi walls are tangent to the spectrahedral approximations. It would be desirable to better understand this fact.

\smallskip

We now shift gears, by allowing ${\bf f} = (f_1,\ldots,f_m)$ to be an arbitrary
tuple of polynomials~in $x = (x_1,\ldots,x_n)$. Fix $d\in \NN$ such that
 $\,\deg(f_i)\leq 2d \,$ for all~$i$.
We will construct a spectrahedral shadow $S^d_{\bf f}(y)\subset \RR^n$ that is contained in $\Vor_X(y)$.
The idea is to perform a change of variables that makes the 
constraints ${\bf f}$ quadratic, and  use the construction  above.

Let $\mathcal A := \{\alpha \!\in\! \NN^n: 0\!<\! \sum_i \alpha_i \!\leq\! d\}$.
This set consists of $N:= \binom{n+d}{d}\!-\!1$ nonnegative integer vectors.
We consider the $d$-th Veronese embedding of affine $n$-space into affine $N$-space:
\begin{align*}
  \nu_d: \RR^n \to \RR^N, \quad
  x = (x_1,\dots,x_n) \;\mapsto\; z = (z_\alpha)_{\alpha\in \mathcal A},
  \quad\text{ where }
  z_\alpha:= x^\alpha = x_1^{\alpha_1}\cdots x_n^{\alpha_n}.
\end{align*}
Among the entries of $z$ are the variables~$x_1,\dots,x_n$.
We list these at the beginning in the vector $z$.
The image of $\nu_d$ is the Veronese variety. 
It is defined by the quadratic equations
\begin{align}\label{eq:veronese}
  z_{\alpha_1}z_{\alpha_2} = z_{\beta_1}z_{\beta_2}
  \qquad\forall 
  \alpha_1,\alpha_2,\beta_1,\beta_2
  \;\text{ such that }\;
  \alpha_1+\alpha_2 = \beta_1+\beta_2.
\end{align}
Since the polynomial $f_i(x)$ has degree $\leq 2d$, there is a quadratic function $q_i(z)$ such that $f_i(x) = q_i(\nu_d(x))$. The Veronese image $\nu_d(X)$ is defined
by the quadratic equations $q_i(z)=0$ together with those in~\eqref{eq:veronese}.
We write ${\bf q}\subset \RR[z]$ for the (finite) set of all of these quadrics.

Each $q\in {\bf q}$ is a quadratic polynomial in $N$ variables.
Let $A_q \in \SymRR^{N}$ be its Hessian matrix.
Let 
$  C {:=} \left(\begin{smallmatrix}
    I_n & 0 \\ 0 & 0
  \end{smallmatrix}\right)
  \in \SymRR^{N} $
be the Hessian of the function $z \mapsto x^T x$.
The master spectrahedron is
\begin{align*}
  {\rm S}_{\bf q} \,\,:= \,\, \{\, \lambda \in \RR^{|\bf q|}\,
  :\, \sum_{q\in {\bf q}} \lambda_q A_q \preceq C \,\}.
\end{align*}
Let $J_{\bf q}(y) := \Jac_{\bf q}(\nu_d(y))$ be the Jacobian matrix of ${\bf q}$ evaluated at the point $\nu_d(y) \in \RR^N$.
This matrix has $N$~rows.
Let $J^n_{\bf q}(y)$ be the submatrix of $J_{\bf q}(y)$ 
that is given by the $n$~rows corresponding to the variables~$x_i$, and let 
$J^{N-n}_{\bf q}(y)$ be the submatrix given by the remaining $N\!-\!n$ rows.
We now consider the spectrahedral shadow
\begin{align*}
  S^d_{\bf f}(y) \,\,: = \,\, y \,-\, \tfrac{1}{2}\, J^n_{\bf q}(y) \,\cdot\, ({\rm S}_{\bf q} \cap \ker J^{N-n}_{\bf q}(y)) \,\,\,\subset \,\,\,\RR^n.
\end{align*}
This is obtained by intersecting
the spectrahedron ${\rm S}_{\bf q}$ with a linear subspace and then taking the image under an affine-linear map.
One can show the inclusions in~\eqref{eq:inclusions}  using  ideas similar to those in Lemma~\ref{thm:inclusions1}.
An alternative argument is given in the proof of
Corollary \ref{thm:inclusions}.

\begin{example}\rm
  Consider the cardioid curve $X= V((x_1^2{+}x_2^2{+}x_1)^2{-}x_1^2{-}x_2^2)$ in~$\RR^2$,
  shown in red in Figure~\ref{fig:cardiod}. See also
  \cite[Figure 1]{DHOST}.
   We compare the Voronoi cells with the spectrahedral relaxation of degree $d=2$.
  The Voronoi cell at the origin, a singular point, is the interior of the circle $C : \{u_1^2{+}u_2^2{+}u_1{=}0\}$.
  The Voronoi cell at a smooth point $y$ is contained in the normal line to $X$ at $y$.
  It is either a ray emanating from the circle $C$, or a line segment from $C$ to the $x$-axis.
  The spectrahedral shadow $S_{\bf f}^2(y)$ is the subset of the Voronoi cell outside of the cardioid.
  For instance, the Voronoi cell at $y=(0,1)$ is the ray 
  $\,\Vor_X(y) = \{(t,t{+}1): t \geq - \frac{1}{2}\}$, and its spectrahedral 
  approximation is the strictly smaller ray
   $\,S_{\bf f}^2(y) = \{(t,t{+}1): t \geq 0\}$.
\end{example}

\begin{figure}[htb]
  \centering
  \includegraphics[width=170pt]{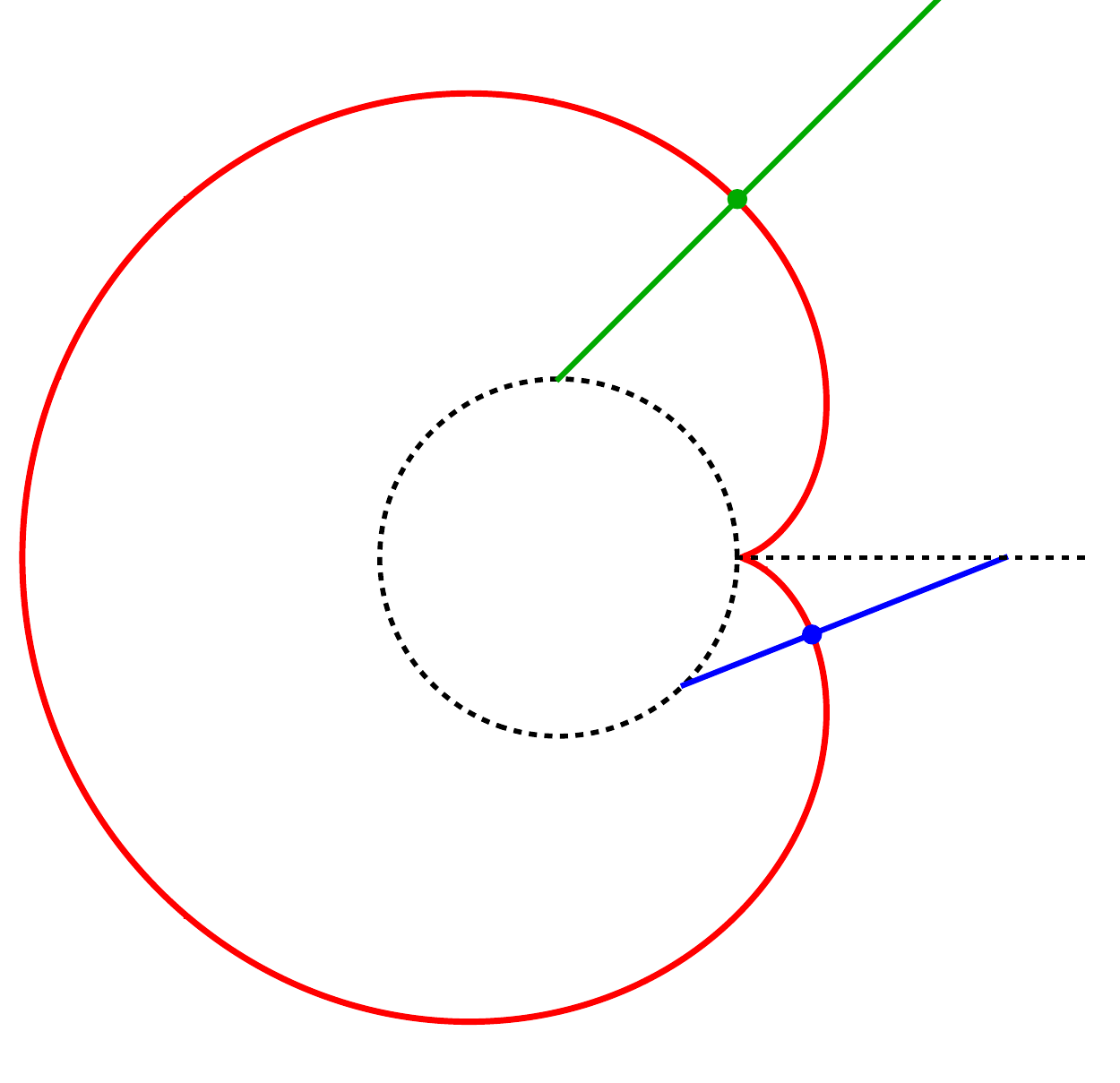}
  \caption{\label{fig:cardiod}  The red curve is the cardioid.
   The inner circle is the Voronoi cell
  of its singular point. Other Voronoi cells are either rays or line segments.
  These emanate from the circle. As they exit the cardioid, they enter the
     spectrahedral approximation to the Voronoi cell.}
  \end{figure}

Fix $u\in \RR^n$, and let $y$ be its nearest point on the variety $X$.
Though computing this nearest point~$y$ is hard in general, 
we can do it efficiently if $u$ lies in the interior of the spectrahedral shadow $S^d_{\bf f}(y)$
for some fixed $d$. Indeed, this is done by solving a certain SDP.

\begin{proposition}
  Consider the $d$-th level of the SOS 
  hierarchy for the optimization problem $\,\min_{x\in V({\bf f})}\|x-u\|^2$.
  A point $u$ lies in the interior of the spectrahedral shadow $S^d_{\bf f}(y)$ if and only if the $d$-th SOS relaxation exactly recovers~$y$ (i.e.~the moment matrix has rank one).
\end{proposition}

\begin{proof}
  The $d$-th level of the relaxation is obtained by taking the Lagrangian dual of the 
  quadratic optimization problem (QCQP) given by the quadrics in the set ${\bf q}$ above;
  see \cite{BPT}.
  The SDP-exact region in quadratic programming was formally defined in~\cite[Definition~3.2]{CHS}.
  It is straightforward to verify that this definition agrees with our description of~${S}_{\bf f}^d(y)$.
\end{proof}

\begin{corollary}\label{thm:inclusions}
  The inclusions
  $S^d_{\bf f}(y) \subset S^{d+1}_{\bf f}(y) \subset \Vor_X(y)$
  hold.
\end{corollary}
\begin{proof}
  If the SOS relaxation recovers a point $u$, then it must lie in the Voronoi cell $ \Vor_X(y)$.
  And if the $d$-th SOS relaxation is exact then the $(d+1)$-st relaxation is also exact.
\end{proof}

\begin{example} \rm
  Consider the problem of finding the nearest point from a point $u\in \RR^2$ to the cardioid.
  By \cite[Example 1.1]{DHOST}, the ED degree is $3$.
    Here we consider the second SOS relaxation of the problem.
  We characterized the sets $S_{\bf f}^2(y)$ above.
  It follows that the second SOS relaxation solves the problem exactly if and only if $u$ lies on the outside of the cardioid.
\end{example}

\section{Formulas for Curves and Surfaces}\label{s:5}

The algebraic boundary of the Voronoi cell $\Vor_X(y)$  
is a hypersurface in the normal space to 
a variety $X \subset \RR^n$ at a point $y\in X$.
We study the degree of that hypersurface
when $X$ is a curve or a surface.
We denote this degree by $\delta_X(y)$ and refer to it as the Voronoi degree.
We identify $X$ and $\partial_{\rm alg} \! \Vor_X(y)$ with their
Zariski closures in complex projective space~$\PP^n$.

\begin{theorem} \label{thm:curves}
Let $X\subset \PP^n$ be a curve of degree~$d$ and geometric genus~$g$ with at most ordinary multiple points as singularities.  The Voronoi degree at a general point $y \in X$ equals
$$ \delta_X(y) \,\,= \,\, 4d+2g-6, $$ 
provided $X$ is in general position in $\PP^n$.
\end{theorem}

\begin{example} \rm
If $X$ is a smooth curve of degree $d$ in the plane, then $2g-2=d(d-3)$, so 
$$\delta_X(y)\,\,= \,\,d^2+d-4.$$
This confirms our experimental results in the row $n=2$ of Table~\ref{tab:voronoi}.
\end{example} 

\begin{example} \rm
If $X$ is a rational curve of degree $d$, then $g=0$  and hence
$\,\delta_X(y)=4d-6$.
If $X$ is an elliptic curve, so the genus is $g=1$, then we have $\delta_X(y) = 4d-4$.
A space curve with $d=4$ and $g=1$  was studied in Example~\ref{ex:spacecurve}.
Its Voronoi degree equals  $\delta_X(y) = 12$.
\end{example}

The proof of Theorem~\ref{thm:curves} appears in the next section. We will then
see  what general position means. For example,
let $X$ be the twisted cubic curve in $\PP^3$, with affine
parameterization $t\mapsto (t,t^2,t^3)$. Here $g=0$ and $d=3$, so the expected
Voronoi degree is $6$. But in Example~\ref{ex:twistedcubic} we saw 
$\delta_X(y) = 4$. This is explained by the fact that the plane 
at infinity in $\PP^3$ intersects the curve $X$ in a triple point. After a
general linear change of coordinates in $\PP^3$, which amounts to
a linear fractional transformation in $\RR^3$, we correctly find $\delta_X(y) = 6$.

We next present a formula for the Voronoi degree of a surface $X$ which is 
smooth and irreducible in $\PP^n$. Our formula is 
in terms of its degree $d$ and two further invariants.  The first, denoted
$\, \chi(X):=c_2(X)$, is the topological Euler characteristic.
This is equal to the degree of
the second Chern class of the tangent bundle. 
The second invariant, denoted $g(X)$, is the genus of the curve  obtained by intersecting $X$ with a general smooth quadratic
hypersurface in $\PP^n$. Thus, $g(X)$ is the quadratic analogue to the 
usual sectional genus  of the surface $X$.

\begin{theorem} \label{thm:surfaces} 
Let $X\subset \PP^n$ be a smooth surface of degree $d$. Then its Voronoi degree equals
$$\delta_X(y)\,\, = \,\, 3d + \chi(X)+4g(X) -11, $$
provided the surface $X$ is in general position in $\PP^n$ and $y$ 
is a general point on $X$.
\end{theorem}

The proof of Theorem~\ref{thm:surfaces} will also be presented in the next section.
At present we do not know how to generalize these formulas to the case when
$X$ is a variety of dimension $\geq 3$.

\begin{example} \rm
If $X$ is a smooth surface in $\PP^3$ of degree $d$, then $\chi(X)= d(d^2-4d+6)$, by \cite[Ex~3.2.12]{FUL}. A smooth quadratic hypersurface section of $X$
is an irreducible curve of degree $(d,d)$ in $\PP^1 \times \PP^1$. The genus of such a curve
is $g(X) = (d-1)^2$. We conclude that
$$\delta_X(y)\,\,\,=\,\,\, 3d  \,+ d(d^2-4d+6)\,+\,4(d-1)^2\,-\,11\,\,\,=\,\,\,d^3+d-7.$$
This confirms our experimental results in the row $n=3$ of Table~\ref{tab:voronoi}.
\end{example}

\begin{example}\label{ex:veronese0} \rm
Let $X$ be the Veronese surface of order $e$ in $\PP^{\binom{e+1}{2}-1}$,
taken after a general linear change of coordinates in that ambient space.
The degree of $X$ equals $d = e^2$.
We have $\,\chi(X)= \chi(\PP^2) =3\,$, and the general quadratic hypersurface section 
 of $X$ is a curve of genus  $\,g(X) =   \binom{2e-1}{2}$.  We conclude that the 
Voronoi degree of $X$ at a general point $y$ equals
$$\delta_X(y)\,\, = \,\, 3 e^2 \,+\,
  3\,+\,2(2e{-}1)(2e{-}2)\,-\,11 \,\,=\,\,\,11e^2-12e-4. $$
For instance, for the quadratic Veronese surface in $\PP^5$ we have $e=2$ and hence
$\,\delta_X(y) = 16$.
This is smaller than the number~$18$ found in Example~\ref{ex:331},
since back then we were dealing with the cone over the Veronese surface in~$\RR^6$,
and not with the Veronese surface in $\RR^5 \subset \PP^5$.
\end{example}

We finally consider affine surfaces
 defined by homogeneous polynomials.
Namely, let $X\subset \RR^n$ be the affine cone over a general smooth curve
of degree $d$ and genus $g$ in~$\PP^{n-1}$.

\begin{theorem} \label{thm:cones} 
Let $X\subset \RR^n$ be the cone over a smooth curve in $\PP^{n-1}$.
Its Voronoi degree is
$$\delta_X(y)\,\, = \,\, 6d + 4g- 9 $$
provided that the curve is in general position and $y$ is a general point.
\end{theorem}

The proof of Theorem~\ref{thm:cones} will be presented in the next section.

\begin{example} \rm 
If $X\subset\RR^3$ is the cone over a smooth curve of degree $d$ in $\PP^2$, then
$2g-2=d(d-3)$. 
Hence the Voronoi degree of $X$ is
$$\delta_X(y)\,\,=\,\,2d^2-5.$$
This confirms our experimental results in the row $n=3$ of Table~\ref{tab:voronoihom}.
\end{example}

To conclude, we comment on the assumptions made in our theorems.
We assumed that the variety $X$ is in general position in $\PP^n$.
If this is not satisfied, then the Voronoi degree may drop.
Nonetheless, the technique introduced in the next section can be adapted to 
determine the correct value.
As an illustration, we consider the affine Veronese surface (Example~\ref{ex:veronese0}).

\begin{example} \label{ex:veronese} \rm
  Let $X\subset \PP^5$ be the Veronese surface with affine parametrization $(s,t)\mapsto (s,t,s^2,st,t^2)$.
  The hyperplane at infinity intersects $X$ in a double conic, so $X$ is not in general position.
  In the next section, we will show that the true Voronoi degree is $\delta_X(y)=10.$
  For the Frobenius norm, the Voronoi degree drops further.
  For this, we shall derive $\delta_X(y)=4$.
\end{example}

\section{Euler Characteristic of a Fibration}\label{s:6}

In this section we develop the geometry and the proofs for the degree formulas in Section~\ref{s:5}.
Let $X\subset \PP^n$ be a smooth projective variety defined over~$\RR$.  We assume that $y\in X$ is a general point, and that we fixed an affine space $\RR^n\subset \PP^n$ containing $y$ such that the hyperplane at infinity $\PP^n\backslash \RR^n$ is in general position with respect to $X$.  We use the Euclidean metric in this $\RR^n$ to define the normal space to $X$ at $y$.  This can be expressed equivalently as follows.
 After a projective transformation in $\PP^n$, we can assume
that $\RR^n = \{x_0=1\}$,  that
  $y=[1:0:\cdots:0]$ is a point in $X$, and that 
  the tangent space to $X$ at $y$ is contained in the hyperplane $\{x_n=0\}$.   
  The normal space to $X$ at $y$ contains the line 
 $\{x_1=\ldots = x_{n-1} = 0\}$.
 
 The sphere  through $y$ with center $[1 :0 : \cdots : 0 : u\,]$ on this normal line~is
 $$ \quad Q_u \,\,= \,\, \{\,x_1^2+\cdots+x_{n-1}^2+(x_n-u)^2-u^2   \,= \,0 \,\}
 \quad \subset \,\,\,\RR^n. $$
As $u$ varies, this is a linear pencil that extends to a $\PP^1$ family of quadric hypersurfaces
$$Q_{(t\,:\,u)} \,\,=\,\, \{\,t(x_1^2+\cdots+x_{n-1}^2+x_{n}^2) -2u x_0 x_n\,=\,0 \,\} 
\quad \subset \,\, \PP^n. $$
Note that $Q_{(t\,:\,u)}$ is tangent to $X$ at $y$.
Assuming the normal line to be general, we observe:

\begin{remark}  \label{rem:count} \rm
The  Voronoi degree $\,\delta_X(y)\,$ is the number of quadratic hypersurfaces  $Q_{(t\,:\,u)}$ with $t\not= 0$ that are tangent to $X$ at a point in the affine space $\{x_0=1\}$ distinct from~$y$.
\end{remark}

We shall compute this number by counting tangency points of all quadrics in the pencil.  In particular we need to consider the special quadric $Q_{(0:1)}=  \{x_0 x_n=0\} $.
This  quadric is reducible: it consists of the tangent hyperplane $\{x_n=0\}$ and the hyperplane at infinity $\{x_0=0\}$.
It is singular along a codimension two linear space $\{x_0=x_n=0\}$ at infinity.
Any point of $X$ on this linear space is therefore also a point of tangency between $X$ and $Q_{(0:1)}$.

To count the tangent quadrics, we consider
the map $\,X\dashrightarrow \PP^1,\,   x\mapsto (t:u)\,$ whose fibers are 
the intersections $X_{(t:u)} = Q_{(t:u)}\cap X$.
By Remark~\ref{rem:count}, we need to
count its ramification points. However,
this map is not a morphism. Its base locus is
$X\cap Q_{(1:0)}\cap Q_{(0:1)} $. We blow up that base locus to get a
morphism which has the intersections $X_{(t:u)} $ for its fibers:
$$ q: \tilde X\to \PP^1.$$
The topological Euler characteristic (called {\em Euler number}) of the
 fibers of $q$ depends on the singularities. We shall count the tangencies 
 indirectly, by computing the  Euler number
$\chi(\tilde X)$  of the blow-up $\tilde X$ in two ways, first directly as a variety, 
and secondly as a fibration over $\PP^1$.

Euler numbers have the following two fundamental properties. The first 
property is {\em multiplicativity}. It
 is found in topology books, e.g.~\cite[Chapter 9.3]{spanier}.
Namely, if $W\to Z$ is a surjection of topological spaces, $Z$ is connected and all fibers are homeomorphic to a topological space $Y$, then $\chi(W)=\chi(Y)\cdot \chi(Z)$.  The second property is {\em additivity}. It applies to complex varieties, 
as seen in  \cite[Section 4.5]{FulT}.
To be precise, if $Z$ is a closed algebraic subset of a 
complex variety $W$ with complement $Y$, then $\chi(W)=\chi(Y)+\chi(Z)$.

\smallskip

For the fibration $ q: \tilde X\to \PP^1$, the first property may be applied to the set of fibers that are smooth, hence homeomorphic, while the second may be used when adding the singular fibers.  Assuming that singular fibers (except the special one) have a quadratic node as its singular point, the Voronoi degree $v:=\delta_X(y)$ satisfies the equation
\begin{align}\label{eq:euler}
\chi(\tilde X)&\,\,=\,\,(1-v)\,\chi(X_{gen})\,+\,\chi(X_{(0:1)})\,+ \,v\;\chi(X_s). 
\end{align}
Here $X_{gen}$ is a smooth fiber of the fibration, $X_{(0:1)}$ is the special fiber over $(0:1)$, and $X_s$ is a fiber with one quadratic node as singular point.   The factor $1-v$ is the 
Euler~number of $\PP^1\,\backslash \{v{+}1\; {\rm points}\}$.
We will use~\eqref{eq:euler} to derive the degree formulas from Section 5, and refer to \cite[Ex 3.2.12-13]{FUL} for Euler numbers of smooth curves, surfaces and hypersurfaces.

\begin{proof}[Proof of Theorem~\ref{thm:curves}]
Let $\bar X\to X$ be a resolution of singularities.
As above, we assume that $y$ is a smooth point on $X$ and that  $\{x_n=x_0=0\}\cap X=\emptyset$. We may pull back the  pencil of quadrics $Q_{(t:u)}$ to $ \bar X$. 
This gives a map $ \bar q: \bar X \dashrightarrow \PP^1$.
All quadrics in the pencil have multiplicity at least $2$ at $y$,  so we remove the divisor $2[y]$ from each divisor in the linear system $\{X_{(t:u)}\}_{(u:t)\in \PP^1}$.
Thus we obtain a pencil of divisors of degree $2d-2$ on $X$ that defines a morphism 
$q:\bar X\to \PP^1.$
The Euler number of $\bar X$ is $\chi(\bar X)=2-2g$. 
The Euler number of a fiber is now simply the number of points in the fiber, i.e. $\chi(X_{gen})=2d-2$  and $\chi(X_s)=2d-3$ for the singular fibers, the fibers where one point appear with multiplicity two.
Also $\chi(X_{(0:1)})=2d-2$, since $X_{(0:1)}$  consists
of $2d-2$ points.  
Plugging into~\eqref{eq:euler} we get:
$$2-2g \,=\,(1-v)(2d-2)+(2d-2)+v(2d-3)\,=\,4d-4-v. $$
We now obtain Theorem~\ref{thm:curves} by solving for $v$.
\end{proof}

The above derivation can also be seen as an application of the Riemann-Hurwitz formula.

\begin{proof}[Proof of Theorem~\ref{thm:surfaces}]
The curves $\{X_{(t:u)}\}_{(u:t)\in \PP^1}$ have a common intersection. This is our base locus
$ X_{(1:0)}\cap X_{(0:1)} $.  
By B\'ezout's Theorem, the number of intersection points is at most $2\cdot 2\cdot d=4d$. 
All curves are singular at $y$, the general one a simple node, so this point counts with multiplicity $4$ in the intersection.  
We assume that all other base points are simple. We thus have $4d-4$ simple points.
We blow up all the base points, $\pi: \tilde X\to X$, with exceptional curve $E_0$ over $y$ and $E_1,\ldots,E_{4d-4}$ over the remaining base points.
The strict transforms of the curves $X_{(t:u)}$ on  $\tilde X$ are then the fibers of a morphism 
$q:\tilde X\to \PP^1$
for which we apply~\eqref{eq:euler}.

The Euler number $\chi(X)$ equals the degree of the Chern class $c_2(X)$ of the tangent bundle of $X$.
Since $\pi$ blows up $4d-3$ points, there are $4d-3$ points on $X$ that are replaced by $\PP^1$s on $\tilde X$.
Since $\chi(\PP^1)=2$, we get 
$\chi(\tilde X)=\chi(X)+4d-3.$
If the genus of a smooth hyperquadric intersection with $X$ is $g$, the general fiber of $q$ is a smooth curve of genus $g-1$, since it is the strict transform of a curve that is singular at $y$. We conclude that
 $\chi(X_{gen})=4-2g$.

To compute $\chi(X_s)$ we remove first the singular point of $X_s$ and obtain a smooth curve of genus $g-2$ with two points removed.
This curve has Euler number $-2(g-2)$.
Adding the singular point, the additivity of the Euler number yields
$$\chi(X_s)\,=\,-2(g-2)+1\,=\,5-2g.$$
The special  curve $X_{(0:1)}$ has two components, one in the tangent plane that is singular at the point of tangency, and one in the hyperplane at infinity.
Assume that the two components are smooth outside the point of tangency and that they meet transversally, i.e., in $d$ points.
We then compute $ \chi(X_{(0:1)})$, as above, by first removing the $d$ points of intersection to get a smooth curve of genus $g-1-d$ with $2d$ points removed, 
and we next use the addition property to add the $d$ points back.
Thus  
$$\chi(X_{(0:1)})\,\,=\,\,2\,-\,2(g-d-1)\,-\,2d \,+\,d \,\,=\,\, d+4-2g.$$
Substituting into the formula~\eqref{eq:euler} gives
$$\chi(X)+4d-3\,\,=\,\,(1-v)(4-2g) \,+\,d+4-2g\,+\,v(5-2g).$$ 
From this we obtain the desired formula
$\,v=3d+\chi(X)+4g-11.$
\end{proof}

\begin{proof}[Details for Example~\ref{ex:veronese}]
The given Veronese surface~$X$ intersects the hyperplane at infinity in a double conic instead of a smooth curve.
We explain how to compute the Voronoi degree in this case.
For the Euclidean metric, the general curve $X_{(t:u)} = Q_{(t:u)}\cap X$ is transverse at four points on this conic.
Then $X_{(0:1)}$ has seven components, $\chi(X_{(0:1)})=8$ and $\delta_X(y)=10.$

For the Frobenius metric,
the curves $X_{(t:u)}$ are all singular at two distinct points on this conic.
The three common singularities of the curves $X_{(t:u)}$ are part of the base locus of the pencil.
Outside  these three points, the pencil has $4$ additional basepoints, so the 
map $\tilde X\to X$ blows up $7$ points.
Hence $\chi(\tilde X)=10$.
The curve $X_{gen}$ is now rational, so $\chi(X_{gen})=2$ and $\chi(X_{s})=1$.
The reducible curve $X_{(0:1)}$  has two components from the tangent hyperplane, and only the conic from the hyperplane at infinity.
Therefore $X_{(0:1)}$ has three components, all $\PP^1$s, two that are disjoint and one that meets the other two in one point each, and so $\chi(X_{(0:1)})=4.$  
Equation~\eqref{eq:euler} gives
$10 = (1-v)\cdot 2  + 4 + v\cdot 1 $, and hence $v=\delta_X(y)=4.$
\end{proof}

\begin{proof}[Proof of Theorem~\ref{thm:cones}]
The closure of the affine cone $X$ in~$\PP^n$ is a projective surface as above.
We need to blow up also the vertex of the cone to get a morphism from a smooth surface 
$q:\tilde{X}\to \PP^1.$
The Euler number of the blown up cone is $\chi(C)\cdot \chi(\PP^1)=2(2-2g)$, so 
$$\chi(\tilde{X})\,\,= \,\,2(2-2g)+4d-3 \,\,= \,\,1+4d-4g.$$
The genus of a smooth quadratic hypersurface section $X\cap Q$ is then $g(X\cap Q)=2g-1+d$. Hence the strict transform of each one nodal  quadratic hypersurface section $X_{gen}$ has genus $2g-2+d$. The
Euler number equals $\chi(X_{gen})= 4-2d-4g$, while  $\chi(X_{s})= 3-2d-4g$.

The tangent hyperplane at $y$ is tangent to the 
line $L_0$ in the cone through $y$, and it intersects
the surface $X$ in $d-2$ further lines $L_1,\ldots,L_{d-2}$.  Therefore,
$$ X_{(0:1)}\,\,=\,\,C_0+L_0+L_1 +\cdots+L_{d-2}+C_\infty,$$
where $C_0$ is the exceptional curve over the vertex of the cone, $C_\infty$ is the strict transform of the curve at infinity, and  the $L_i$ are the strict transforms of the lines through $y$.  We conclude
\begin{align*}
\chi(X_{(0:1)}) \,\,&=\,\, \chi(C_0\backslash (d-1)\; {\rm points})+\chi(C_\infty\backslash (d-1)\; {\rm points})\,+\,(d-1)\cdot \chi(\PP^1) \\
&=\,\, 2(2-2g-(d-1))+2(d-1) \,\, = \,\, 4-4g.
\end{align*}
From~\eqref{eq:euler} we get
\begin{align*}
\chi(\tilde X) \,\,&=\,\, (1-v)\chi(X_{gen})+\chi(X_{(0:1)})+v\cdot \chi(X_s),  \\
1{+}4d{-}4g \,\,&=\,\,  (1{-}v)\cdot (4{-}2d{-}4g)+4{-}4g+ v\cdot (3{-}2d{-}4g))
\,\,=\,\,  v - 8g - 2d + 10.
\end{align*}
This means that the Voronoi degree is
$\delta_X(y)=v=4g+6d-9.$
\end{proof}

This concludes the proofs of the degree formulas we were able to find for boundaries of Voronoi cells. 
It would, of course, be very desirable to extend these to higher dimensional varieties.
The work we presented is one of the first steps in metric algebraic geometry. 
In additional to addressing some foundational questions, 
natural connections to low rank approximation (Section 3) and convex optimization (Section 4) were highlighted.

\bigskip \bigskip \bigskip

\paragraph{Acknowledgments.}
Research on this project was carried out while the authors 
were based at the Max-Planck Institute
for Mathematics in the Sciences (MPI-MiS) in Leipzig and at the
Institute for Experimental and Computational Research in
Mathematics (ICERM) in Providence. We are grateful to
both institutions for their support. Bernd Sturmfels and Madeleine Weinstein
received additional support from the US National Science Foundation.

\bigskip

\begin{small}

\end{small}

\bigskip \bigskip \bigskip \bigskip

\noindent
\footnotesize {\bf Authors' addresses:}

\smallskip

\noindent Diego Cifuentes, Massachusetts Institute of Technology
\hfill {\tt diegcif@mit.edu}

\noindent Kristian Ranestad, University of Oslo
\hfill {\tt ranestad@math.uio.no}

\noindent Bernd Sturmfels,
 \  MPI-MiS Leipzig and
UC  Berkeley \hfill  {\tt bernd@mis.mpg.de}

\noindent Madeleine Weinstein,
UC  Berkeley \hfill {\tt maddie@math.berkeley.edu}

\end{document}